\documentclass{article}

\usepackage{amssymb,amsmath,color}
\usepackage{amsthm}
\usepackage{url}
\usepackage{tikz}

\definecolor{red}{rgb}{1,0,0}
\definecolor{blue}{rgb}{.2,.2,.8}

\def\m{\mu}
\def\T{\mathcal T}
\def\S{\mathcal S}
\def\C{\mathcal C}
\def\A{\mathcal A}
\def\D{\mathcal D}
\def\OO{\mathcal O} 
\def\MD{\mathcal{MD}}
\def\MDS{\mathcal{DD}}
\def\OD{\overline{\mathcal{D}}}

\newtheorem{theorem}{Theorem}[section]
\newtheorem{corollary}[theorem]{Corollary}
\newtheorem{proposition}{Proposition}

\theoremstyle{definition}
\newtheorem{definition}{Definition}
\newtheorem{example}{Example}
\newtheorem{remark}{Remark}

\newcommand{\ds}{\displaystyle}

\begin{document}

\title{Combinatorial proofs of two Euler type identities due to  Andrews}
\author{Cristina Ballantine
	\and   Richard Bielak
}
\date{}
\maketitle

\begin{abstract} Let $a(n)$ be the number of partitions of $n$ such that the set of even parts has exactly one element, $b(n)$ be the difference between the number of parts in all odd partitions of $n$ and the number of parts in all distinct partitions of $n$, and $c(n)$ be the number of partitions of $n$ in which exactly one part is repeated. Beck conjectured that $a(n)=b(n)$ and Andrews, using generating functions,  proved that $a(n)=b(n)=c(n)$. We give a combinatorial proof of Andrews' result. Our proof relies on  bijections between a set and a multiset, where the partitions in the multiset are decorated with bit strings. We  prove combinatorially Beck's second identity, which was also proved by Andrews using generating functions. Let $c_1(n)$ be the number of partitions of $n$ such that there is exactly one part occurring three times while all other parts occur only once and let $b_1(n)$ to be the difference between the total number of parts in the partitions of $n$ into distinct parts and the total number of different parts in the partitions of $n$ into odd parts. Then, $c_1(n)=b_1(n)$. 
\end{abstract}

\section{Introduction}

In \cite{A17}, following conjectures of Beck, Andrews  considered what happens if one relaxes the conditions on parts in  Euler's partition identity. He gave analytic proofs of related identities. In this article, we provide combinatorial proofs of his results and add one more special case. 

Given a non negative integer $n$, a \textit{partition} $\lambda$ of $n$ is a non increasing sequence of positive integers $\lambda=(\lambda_1, \lambda_2, \ldots \lambda_k)$ that add up to $n$, i.e., $\ds\sum_{i=1}^k\lambda_i=n$. The numbers $\lambda_i$ are called the \textit{parts} of $\lambda$ and $n$ is called the \textit{size} of $\lambda$. The number of parts of the partition is called the \textit{length} of $\lambda$ and is denoted by $\ell(\lambda)$.

\medskip

Let $\OO(n)$ be the set of partitions of $n$ with all parts odd and let $\D(n)$ be the set of partitions of $n$ with distinct parts. Then, Euler's identity states that $|\OO(n)|=|\D(n)|$. 

Let $\A(n)$ be the set of partitions of $n$ such that the set of even parts has exactly one element. Let $\C(n)$ be the set of partitions of $n$ in which exactly one part is repeated. 

Let $a(n)=|\A(n)|$ and $c(n)=|\C(n)|$. Let $b(n)$ be the difference between the number of parts in all odd partitions of $n$ and the number of parts in all distinct partitions of $n$. Thus, $b(n)$ is the difference between the number of parts in all partitions in $\OO(n)$ and the number of parts in all partitions in $\D(n)$.  In \cite{B1, B2}, Beck conjectured that $a(n)=b(n)=c(n)$.  Andrews proved these identities in \cite{A17}  using generating functions.   In \cite{FT17}, Fu and Tang gave two generalizations of this result. For one of the generalizations, Fu and Tang give a combinatorial proof and, as a particular case, they obtain a combinatorial proof for $a(n)=c(n)$.  We give a combinatorial proof for the identities involving $b(n)$. 

\begin{theorem} \label{T1} Let $n\geq 1$. Then,  \medskip

(i) $a(n)=b(n)$

\medskip

(ii) $c(n)=b(n)$. \end{theorem} The novelty of our approach is the use of partitions decorated by bit strings. This allows us to create bijections between a set of partitions and a multiset of partitions. We distinguish the partitions in the multiset via decorations with bit strings. 

\begin{definition} A \textit{bit string} $w$ is a sequence of letters from the alphabet $\{0,1\}$. The \textit{length} of a bit string $w$, denoted $\ell(w)$, is the number of letters in $w$. We refer to \textit{position} $i$ in $w$ as the $i$th entry from the right, where the most right entry is counted as position $0$. 
\end{definition}

Note that leading zeros are allowed and are recorded. Thus $010$ and $10$ are different bit strings even though they are the binary representation of the same number. We have $\ell(010)=3$ and $\ell(10)=2$. The empty bit string has length $0$ and is denoted by $\emptyset$. \medskip

Let $\T(n)$ be the subset of $\C(n)$ consisting of partitions of $n$ in which  one part is repeated exactly three times and all other parts occur only once. Let $c_1(n)=|\T(n)|$. Let $b_1(n)$ to be the difference between the total number of parts in the partitions of $n$ into distinct parts and the total number of \textit{different} parts in the partitions of $n$ into odd parts. Thus, $b_1(n)$ is the difference between the number of parts in all partitions in $\D(n)$ and the number of different parts in all partitions in $\OO(n)$ (i.e., parts  counted without multiplicity). We will prove combinatorially the following Theorem.

\begin{theorem} \label{T2} Let $n\geq 1$. Then,  $c_1(n)=b_1(n)$.
\end{theorem}

After the work for this paper was finished we found out that Jane Y.X. Yang has proved these results in greater generality \cite{Y18}. However, our approach using decorations with bit strings allows us to extent Theorem \eqref{T2} to an Euler type identity in theorem \eqref{T3}. 
 
\section{Combinatorial proofs of Theorems \ref{T1} and \ref{T2}}
\subsection{$b(n)$ as the cardinality of a multiset of partitions}\label{b}

First, we recall Glaisher's bijection $\varphi$ used to prove Euler's identity. It is the map from the set of partitions with odd parts to the set of partitions with distinct maps which merges equal parts repeatedly. 

\begin{example} For $n=7$, Glaisher's bijection is given by 

\begin{center}
\begin{tabular}{ccc}
$(7)$ & $\stackrel{\varphi}{\longrightarrow}$ & $(7)$\\ \ \\ $(5, \underbrace{1,1})$ &$\longrightarrow$ & $(5,2)$ \\ \ \\ $(\underbrace{3,3},1)$ &$\longrightarrow$ & (6,1)\\ \ \\ $(3,\underbrace{\underbrace{1,1,}\underbrace{1,1}})$ & $\longrightarrow$ & $(4,3)$\\ \ \\ $(\underbrace{\underbrace{1,1,} \underbrace{1,1,}} \underbrace{1,1,} 1)$&  $\longrightarrow$ & $(4,2,1)$
\end{tabular}
\end{center}
\end{example}
Thus, each partition  $\lambda\in \OO(n)$, has at least as many parts as its image $\varphi(\lambda)\in \D(n)$.   

When calculating $b(n)$, the difference between the number of parts in all odd partitions of $n$ and the number of parts in all distinct partitions of $n$, we sum up the differences in the number of parts in each pair  ($\lambda,\varphi(\lambda))$. Write each part $\mu_j$ of  $\mu=\varphi(\lambda)$ as  $\mu_j=2^{k_j}\cdot m_j$ with $m_j$ odd. Then, $\mu_j$ was obtained by merging $2^{k_j}$ parts in $\lambda$ and thus contributes an excess of $2^{k_j}-1$ parts to the difference. Therefore, the difference between the number of parts of $\lambda$ and the number of parts of $\varphi(\lambda)$ is $\ds \sum_{j=1}^{\ell(\varphi(\lambda))} (2^{k_j}-1)$.

\begin{definition} Given a partition $\mu$ with parts $\mu_j=2^{k_j}\cdot m_j$, where $m_j$ odd, the \textit{weight} of the partition is $wt(\mu)=\ds \sum_{j=1}^{\ell(\mu)} (2^{k_j}-1)$.
\end{definition}
Thus, $wt(\mu)>0$ if and only if $\mu$ contains at least one even part. 

We denote by $\MD(n)$ the multiset of partitions of $n$ with distinct parts in which every partition $\mu\in\D(n)$ appears exactly $wt(\mu)$ times. For example, $wt(4,2,1)=4$  and $(4,2,1)$ appears four times in $\MD(7)$.  Since $wt(7)=0$, the partition $(7)$ does not appear in $\MD(7)$.

The discussion above proves the following interpretation of $b(n)$. 

\begin{proposition} Let $n\geq 1$. Then, $b(n)=|\MD(n)|$.  

\end{proposition}

 To create  bijections from $\A(n)$ to $\MD(n)$ and from $\C(n)$ to $\MD(n)$, we need to distinguish identical elements of $\MD(n)$ and thus view it as a set. Recall that all partitions in $\MD(n)$ have distinct parts and at least one even part.  

\begin{definition} A \textit{decorated partition} is a  partition $\mu$ with  at least one even part  in which  one single \textit{even} part, called the decorated part, has a bit string $w$ as an index.   If the decorated part is $\mu_i=2^km$, where $k\geq 1$ and $m$ is odd, the index  $w$ has  length  $0\leq\ell(w)\leq k-1$. \end{definition} Since there are $2^t$ distinct bit strings of length $t$, there are $2^k-1$ distinct bit strings $w$ of length $0\leq\ell(w)\leq k-1$. Thus, for each even part $\mu_i=2^km$ of $\mu$ there are $2^k-1$ possible indices and for each partition $\mu$ there are precisely $wt(\mu)$ possible decorated partitions with the same parts as $\mu$. \medskip

We denote by $\MDS(n)$ the set of decorated partitions of $n$ with distinct parts. Note that, by definition, a decorated partition has at least one even part. Then, $$|\MD(n)|=|\MDS(n)|$$ and therefore $$b(n)=|\MDS(n)|.$$

\subsection{A combinatorial proof for $a(n)=b(n)$}

We prove that $a(n)=b(n)$ by establishing a one-to-one correspondence between $\A(n)$ and $\MDS(n)$.\medskip

\noindent \textit{From $\MDS(n)$ to $\A(n)$:}\label{part i}

Start with a decorated partition $\mu \in \MDS(n)$. Suppose part $\mu_i=2^k m$, with $k\geq 1$ and $m$ odd, is decorated with bit string $w$ of length $\ell(w)$. Then, $0\leq\ell(w)\leq k-1$. Let $d_w$ be the decimal value of of $w$. We set $d_\emptyset=0$. \medskip

The decorated part will split into two kinds of parts: $2^j m$ and $m$. The length of $w$, $\ell(w)$, determines the size of $2^jm$ and  $d_w$ determines the number of parts of size $2^jm$ that split from the decorated part. All even parts with the same odd part that are larger than the decorated part split into parts equal to $2^jm$. All even parts with the same odd part that are less than the decorated part split into parts equal to $m$. Each of the other even parts splits completely into odd parts. We describe the precise algorithm below.\medskip

1. Split part $\mu_i$ into $d_w+1$ parts of size $2^{k-\ell(w)}m$ and parts of size $m$. Thus, there will be $2^k-(d_w+1)2^{k-\ell(w)}$ parts of size $m$. 
Since $d_w+1\leq 2^{\ell(w)}$,  the resulting number of parts equal to $m$ is non-negative. Moreover,  after the split there is at least one even part. \medskip

2. Every part of size $2^tm$, with $t>k$ (if it exists), splits completely into parts of size $2^{k-\ell(w)}m$, i.e., into $2^{t-k+\ell(w)}$ parts of size $2^{k-\ell(w)}m$. 
\medskip

3. Every other even part splits into odd parts of equal size. I.e., every  part $2^u v$ with $v$ odd, such that $2^u v\neq 2^sm$ for some $s\geq k$, splits into $2^u$ parts of size $v$. \medskip

The resulting partition, $\lambda$ is in $\A(n)$. Its set of even parts is $\{2^{k-\ell(w)}m\}$.

\begin{example} Consider the decorated partition $$\mu=(96,35,34, 24_{01},6,2)=(2^5\cdot 3, 35, 2\cdot 17, (2^3\cdot 3)_{01}, 2\cdot 3, 2)\in \MDS(197).$$
We have $k=3, m=3, \ell(w)=2, d_w=1$. \medskip

1. Part $24=2^3\cdot 3$ splits into two parts of size $6$ and four parts of size $3$. \medskip

2. Part $96=2^5\cdot 3$ splits into $16$ parts of size $6$. \medskip

3. All other even parts split into odd parts. Thus, part $34$ splits into two parts of size $17$,  part $6$ splits into two parts of size $3$, and part $2$ splits into two parts of size $1$. \medskip

The resulting partition is $\lambda=(35, 17^2, 6^{18}, 3^6, 1^2)\in\A(197)$.\medskip

Similarly, the transformation maps the decorated partition \\ $(96,35,34, 24_{10},6,2)\in \MDS(197)$ to  $(35, 17^2, 6^{19}, 3^4, 1^2)\in\A(197)$.

\end{example}

\noindent \textit{From $\A(n)$ to $\MDS(n)$:}

Start with partition $\lambda\in \A(n)$. Then there is one and only one even number $2^k m$, $k\geq 1, m$ odd,  among the parts of $\lambda$. Let $f$ be the multiplicity of the even part in $\lambda$. As in Glaisher's bijection, we merge equal parts repeatedly until we obtain a partition $\mu$with distinct parts. Since $\lambda$ has an even part, $\mu$will also have an even part. 

Next, we determine the decoration of $\mu$. Consider the parts $\m_{j_i}$ of the form $2^{r_i} m$, with $m$ odd and  $r_i\geq k$. We have $j_1<j_2<\cdots$. For notational convenience, set $\m_{j_0}=0$. Let $h$ be the positive integer such that 
\begin{equation}\label{ineq-h}\sum_{i=0}^{h-1}\m_{j_i}<f\cdot 2^k m \leq \sum_{i=0}^{h}\m_{j_i}.\end{equation} Then, we will decorate part $\m_{j_h}=2^{r_h}m$.

To determine the decoration, let $N_h$ be the number of parts $2^km$ in $\lambda$ that merged to form all parts of the form $2^rm>\m_{j_h}$. 
Thus,  $$N_h=\frac{\ds\sum_{i=0}^{h-1}\m_{j_i}}{2^km}.$$ Then, \eqref{ineq-h} becomes $$2^kmN_h<f\cdot 2^km\leq 2^kmN_h+2^{r_h}m,$$ which in turn
implies $N_h<f\leq N_h+2^{r_h-k}$. Therefore, $0<f-N_h\leq 2^{r_h-k}$.

Let $d=f-N_h-1$ and $\ell=r_h-k$. We have $0\leq \ell \leq r_h-1$. Consider the binary representation of $d$ and insert leading zeros to form a bit string $w$ of length $\ell$. Decorate $\m_{j_h}$ with $w$. The resulting decorated partition is in $\MDS(N)$. 

\begin{example} Consider the partition $\lambda=(35, 17^2, 6^{18}, 3^6, 1^2)\in\A(197)$. We have $k=1, m=3, f=18$. Glaisher's bijection produces the partition $\mu=(96,35,34, 24,6,2)\in \MD(197)$. The parts of the form $2^{r_i}\cdot 3$ with $r_i\geq 1$ are $96,24,6$. Since $96<18 \cdot 6 \leq 96+24$ the decorated part will be $24=2^3\cdot 3$. We have $N_h=96/6=16$. To determine the decoration, let $d=18-16-1=1$ and $\ell=3-1=2$. The binary representation of $d$ is $1$. To form a bit string of length $2$ we introduce one leading $0$. Thus, the decoration is $w=01$ and the resulting decorated partition is $(96,35,34, 24_{01},6,2)\in \MDS(197)$.
\medskip
Similarly, starting with  $(35, 17^2, 6^{19}, 3^4, 1^2)\in\A(197)$, after applying Glaisher's bijection we obtain $\mu=(96,35,34, 24,6,2)\in \MD(197)$. All parameters are the same as in the previous example with the exception of $f=19$. As before, the decorated part is $24$ and $\ell=2$. We have $d=19-16-1=2$ whose binary representation is $10$ and already has length 2. Thus $w=10$ and the resulting decorated partition is $(96,35,34, 24_{10},6,2)\in \MDS(197)$.

\end{example}

\subsection{A combinatorial proof for $c(n)=b(n)$}

We could compose the bijection of section \ref{part i} with the bijection of \cite{FT17} (for $k=2$) to obtain a combinatorial proof of part (ii) of Theorem \ref{T1}. We give an alternative  proof 
 that $c(n)=b(n)$ by establishing a one-to-one correspondence between $\C(n)$ and $\MDS(n)$ that does not involve the bijection of \cite{FT17}.\medskip

\noindent \textit{From $\MDS(n)$ to $\C(n)$:}

Start with a decorated partition $\mu \in \MDS(n)$. Suppose part $\mu_i=2^k m$, with $k\geq 1$ and $m$ odd, is decorated with bit string $w$ of length $\ell(w)$ and decimal value $d_w$. Then, $0\leq\ell(w)\leq k-1$. \medskip

  If $w=\emptyset$, split part $\mu_i$  into two equal parts of size $2^{k-\ell(w)-1}m$. The partition $\lambda$ obtained in this way from $\mu  $ is in $\C(n)$. If $w\neq\emptyset$ we obtain $\lambda$ from $\mu$ by performing the following steps. \medskip
  
  1. Split $\mu_i$ into $2(d_w+1)$ parts of size $2^{k-\ell(w)-1}m$ and a part of size $2^{i+k-\ell(w)}m$ for each $i$ such that there is a $0$ in position $i$  in $w$. \medskip
  
  2. Each part $2^tm$ with $k-\ell(w)-1\leq t<k$ splits completely  into parts of size $2^{k-\ell(w)-1}m$, i.e., into $2^{t-k+\ell(w)+1}$ parts of size $2^{k-\ell(w)-1}m$.
\medskip

Since $2(d_w+1)\geq 2$, the obtained partition $\lambda$ is in $\C(n)$. The repeated part is $2^{k-\ell(w)-1}m$. 

\medskip

\noindent {\bf Note:} In step 1. above, after splitting off $2(d_w+1)$ parts of size $2^{k-\ell(w)-1}m$ from $\mu_i$, we are left with $r=2^{k-\ell(w)}(2^{\ell(w)}-d_w-1)m$ to split into distinct parts. We do this by using  Glaisher's transformation $\varphi$ on $r/m$ parts equal to $m$.  Thus, after the splitting, we will have a part equal to $2^jm$ if and only if the binary representation of $r/m$ has a $1$ in position $j$. However, $2^{\ell(w)}-1$ is a bit string of length $\ell(w)$ with every entry equal to $1$. Then, the binary representation of $2^{\ell(w)}-d_w-1$ (filled with leading zeros if necessary to create a bit string of length $\ell(w)$) is precisely the complement of $w$, i.e., the bit string obtained from $w$ by replacing every $0$ by $1$ and every $1$ by $0$. 

\begin{example} Consider the decorated partition $$\mu=( 768, 384_{0110}, 105, 96, 25, 12, 9, 6,2)\ \ \ \ \ \ \ \ \ \ \ \ \ \ \ \ \ \ \ \ \ \ \ \ \  \ \ \ \ \ \ \ \ \ \ \ $$ $$\ \ \ \ \ =(2^8\cdot 3, (2^7\cdot 3)_{0110}, 105, 2^5\cdot 3, 25, 2^2\cdot 3, 9, 2\cdot 3, 2\cdot 1)\in\MDS(1407). $$ We have $k=7, w=0110, \ell(w)=4, d_w=6$. The decorated part is $\mu_2$.  

1. Since $2(d_w+1)=14$ and $w$ has zeros in positions $0$ and $3$, $\mu_2$ splits into $14$ parts of size $2^2\cdot 3$ and one part each of sizes $2^3\cdot 3$ and $2^6\cdot 3$. 

2. The parts of the form $2^t\cdot 3$ with $2\leq t<7$ are $\mu_4=2^5\cdot 3$ and $\mu_6=2^2\cdot 3$. Then, $\mu_4$ splits into $2^3$ parts of size $2^2\cdot 3$ and $\mu_6$ "splits" into one part of size $2^2\cdot 3$.

We obtain the partition $$\lambda=(2^8\cdot 3, 2^6\cdot 3, 105, 25, 2^3\cdot 3, \underbrace{2^2\cdot 3, \ldots, 2^2\cdot 3}_{23 \mbox{ times }}, 9, 2\cdot 3, 2 \cdot 1)$$ $$=(768,192, 105, 25, 24, 12^{23}, 9, 6, 2)\in \C(n).$$
\end{example}\medskip

\noindent \textit{From $\C(n)$ to $\MDS(n)$:}

Start with partition $\lambda\in \C(n)$. Then there is one and only one repeated part  among the parts of $\lambda$. Suppose the repeated part is  $2^k m$, $k\geq 0, m$ odd,  and denote by  $f\geq 2$ its  multiplicity in $\lambda$. As in Glaisher's bijection, we merge equal parts repeatedly until we obtain a partition $\mu$with distinct parts. Since $\lambda$ has a repeated part, $\mu$will  have at least one  even part.

Next, we determine the decoration of $\mu$.  In this case, we want to work with the parts of $\mu$ from the right to the left (i.e., from smallest to largest part). Let $\tilde{\mu}_q=\mu_{\ell(\mu)-q+1}$. 
 Consider the parts $\tilde\m_{j_i}$ of the form $2^{r_i} m$, with $m$ odd and  $r_i\geq k$. We have $j_1<j_2<\cdots$. 

As before, we set $\tilde{\m}_{j_0}=0$. Let $h$ be the positive integer such that 
\begin{equation}\label{ineq1-h}\sum_{i=0}^{h-1}\tilde\m_{j_i}<f\cdot 2^k m \leq \sum_{i=0}^{h}\tilde\m_{j_i}.\end{equation} Then, we will decorate part $\tilde\m_{j_h}=2^{r_h}m$.

To determine the decoration, let $N_h$ be the number of parts $2^km$ in $\lambda$ that merged to form all parts of the form $2^rm<\tilde\m_{j_h}$.
Thus,  $$N_h=\frac{\ds\sum_{i=0}^{h-1}\tilde\m_{j_i}}{2^km}.$$ Then, \eqref{ineq1-h} becomes $$2^kmN_h<f\cdot 2^km\leq 2^kmN_h+2^{r_h}m,$$ which in turn
implies $N_h<f\leq N_h+2^{r_h-k}$. Therefore, $0<f-N_h\leq 2^{r_h-k}$.

Let $\ds d=\frac{f-N_h}{2}-1$ and $\ell=r_h-k-1$. We have $0\leq \ell \leq r_h-1$. Consider the binary representation of $d$ and insert leading zeros to form a bit string $w$ of length $\ell$. Decorate $\tilde\m_{j_h}$ with $w$. The resulting decorated partition (with parts written in non-increasing order) is in $\MDS(N)$. \medskip

\noindent {\bf Note:} To see that $f-N_h$ above is always even consider the three cases below.\medskip 

(i) If $h=1$, then $N_h=0$. In this case we must have had $f=2$. Thus, $f-N_h$ is even. \medskip

(ii) If $f$ is odd, then after the merge we have one part equal to $2^km$ contributing to $N_h$. All other parts contributing to $N_h$ are divisible by $2\cdot 2^km$.  Thus, $N_h$ is odd and   $f-N_h$ is even. 

\medskip

(iii) If $f$ is even and at least $2$, then after the merge we have no part equal to $2^km$ contributing to $N_h$. All  parts contributing to $N_h$ are divisible by $2\cdot 2^km$.  Thus, $N_h$ is even and   $f-N_h$ is even.

\begin{example}Consider the partition $$\lambda=(2^8\cdot 3, 2^6\cdot 3, 105, 25, 2^3\cdot 3, \underbrace{2^2\cdot 3, \ldots ,2^2\cdot 3}_{23 \mbox{ times }}, 9, 2\cdot 3, 2 \cdot 1)$$ $$=(768,192, 105, 25, 24, 12^{23}, 9, 6, 2)\in \C(n).$$ We have $k=2$ and $f=23$. Glaisher's bijection transforms $\lambda$ as follows. 

$$(2^8\cdot 3, 2^6\cdot 3, 105, 25, 2^3\cdot 3, \underbrace{2^2\cdot 3, \ldots ,2^2\cdot 3}_{23 \mbox{ times }}, 9, 2\cdot 3, 2 \cdot 1)$$ $$\downarrow$$ $$ (2^8\cdot 3, 2^6\cdot 3, 105, 25, \underbrace{2^3\cdot 3, \ldots ,2^3\cdot 3}_{12 \mbox{ times }}, 2^2\cdot 3, 9, 2\cdot 3, 2 \cdot 1)$$ $$\downarrow$$ $$ (2^8\cdot 3, 2^6\cdot 3, 105, \underbrace{2^4\cdot 3, \ldots ,2^4\cdot 3}_{6 \mbox{ times }},25,  2^2\cdot 3, 9, 2\cdot 3, 2 \cdot 1)$$ $$\downarrow$$ $$ (2^8\cdot 3, 2^6\cdot 3, 105, 2^5\cdot 3,2^5\cdot 3 ,2^5\cdot 3,25,  2^2\cdot 3, 9, 2\cdot 3, 2 \cdot 1)$$ $$\downarrow$$ $$ (2^8\cdot 3,  2^6\cdot 3, 2^6\cdot 3, 105, 2^5\cdot 3, 25,  2^2\cdot 3, 9, 2\cdot 3, 2 \cdot 1)$$  $$\downarrow$$ $$ \mu = (2^8\cdot 3,  2^7\cdot 3, 105, 2^5\cdot 3, 25,  2^2\cdot 3, 9, 2\cdot 3, 2 \cdot 1)$$

The parts of the form $2^{r}\cdot 3$ with $r\geq 2$ are $\tilde \mu_4=2^2\cdot 3=12, \tilde \mu_6=2^5\cdot 3=96, \tilde\mu_8=2^7\cdot 3=384,$ and $\tilde\mu_9=2^8\cdot 3=768$. Since $12+96<23\cdot2^2\cdot 3\leq 12+96+384$, the decorated part will be $2^7\cdot 3=384$. We have $h=3$ and $N_3=\ds\frac{2^2\cdot 3+2^5\cdot 3}{2^2\cdot 3}=1+2^3=9$. Thus $d=\ds\frac{23-9}{2}-1=6$ and $\ell=7-2-1=4$. Thus $w=0110$ and the resulting decorate partition is 

$$\mu=(2^8\cdot 3, (2^7\cdot 3)_{0110}, 105, 2^5\cdot 3, 25, 2^2\cdot 3, 9, 2\cdot 3, 2\cdot 1)$$ $$= ( 768, 384_{0110}, 105, 96, 25, 12, 9, 6,2)\in\MDS(1407).\ $$

\end{example}

\subsection{$b_1(n)$ as the cardinality of a set of overpartitions} \label{b1}

As in section \ref{b}, we use Glaisher's bijection and calculate $b_1(n)$ by summing up the difference between the number of parts of $\varphi(\lambda)$ and the number of different parts of $\lambda$ of each partition $\lambda\in \OO(n)$. In $\varphi(\lambda)$ there is be a part of size $2^im$, with $m$ odd if and only if there is an $1$ is position $i$ of the binary representation of the multiplicity of $m$ in $\lambda$. After the merge, each odd part in $\lambda$ creates as many parts in $\varphi(\lambda)$ as the number of $1$s in the binary representation of its multiplicity. Moreover, if we write the parts of $\varphi(\lambda)$ as $2^{k_i}m_i$ with $m_i$ odd and $k_i \geq 0$, all parts $2^sm$ with the same largest odd factor $m$ are obtained by merging parts equal to $m$ in $\lambda$. 

For each positive odd integer $2j-1$, denote by $oddm(2j-1)$ the number of parts of $\varphi(\lambda)$ of the form $2^s(2j-1)$ for some $s\geq 0$. Then, given $\lambda \in \OO(n)$,  the difference between the number of parts of $\varphi(\lambda)$ and the number of different parts of $\lambda$ equals $$\sum_{\stackrel{j}{ oddm(2j-1)\neq 0}} (oddm(2j-1)-1).$$

Let $\OD(n)$ be the set of overpartitions of $n$ with distinct parts in which exactly one part is overlined. Part $2^sm$ with $s \geq 0$ and $m$ odd may be overlined only if there is a part $2^tm$ with $t<s$. In particular, no odd part can be overlined. By an overpartition with distinct parts we mean that all parts have multiplicity one, In particular,  $p$ and $\bar p$ cannot not both appear as parts of the overpartition. The discussion above proves the following interpretation of $b_1(n)$.

\begin{proposition} Let $n\geq 1$. Then, $b_1(n)=|\OD(n)|.$ \end{proposition}

\subsection{A combinatorial proof for $c_1(n)=b_1(n)$}

\noindent \textit{From  $\OD(n)$ to $\T(n)$:} 

Start with an overpartition $\mu \in \OD(n)$. Suppose the overlined part is $\mu_i=2^sm$ for some $s\geq1$ and $m$ odd. Then there is a part $\mu_j=2^tm$ of $\mu$ with $t<s$. Let $k$ be the largest positive integer such that $s^km$ is a part of $\mu$ and $k<s$. To obtain $\lambda \in \T(n)$ from $\mu$, split $\mu_i$ into two parts equal to $2^km$ and one part equal to $2^jm$ whenever there is a $1$ in position $j$ of the binary representation of $(2^s-2^{k+1})$, i.e, one part equal to $2^jm$ for each $j=k+1, k+2, \ldots, s-1$.

\begin{example} Let $$\mu=(\overline{768}, 48, 46, 9, 6, 5, 2)=(\overline{2^8\cdot 3},  2^4\cdot 3, 2\cdot 23, 9, 2\cdot 3, 5, 2\cdot 1)\in \OD(884).$$ Then $2^8\cdot 3$ splits into two parts equal to $2^4\cdot 3$ and one part each of size $2^5\cdot 3, 2^6\cdot 3, 2^7\cdot 3$. Thus, we obtain the partition $$\lambda=(2^7\cdot 3, 2^6\cdot 3, 2^5\cdot 3, 2^4\cdot 3, 2^4\cdot 3, 2^4\cdot 3, 2\cdot 23, 9, 2\cdot 3, 5, 2\cdot 1)$$ $$=(384, 192, 96, 48^3, 46, 9, 6, 5, 2)\in \T(884).\ \ \ \ \ \ \ \ \ \ \ \ \ \ \ \ \ \ \ \ \ \ \  $$ \medskip

Similarly, $(768, \overline{48}, 46, 9, 6, 5, 2)=(2^8\cdot 3,  \overline{2^4\cdot 3}, 2\cdot 23, 9, 2\cdot 3, 5, 2\cdot 1)\in \OD(884)$ transforms into $(2^8\cdot 3, 2\cdot 23, 2^3\cdot 3,  2^2\cdot 3,  9, 2\cdot 3, 2\cdot 3, 2\cdot 3, 5, 2\cdot 1)=$\\ $(768, 46, 24, 12, 9, 6^3, 5, 2)\in \T(884)$. 
\end{example}

\noindent \textit{From $\T(n)$ to $\OD(n)$:}

Start with a partition $\lambda \in \T(n)$. Merge the parts of $\lambda$ repeatedly using Glaisher's bijection $\varphi$ to obtain a partition $\mu$ with distinct parts. Overline the smallest part of $\mu$ that is not a part of $\lambda$. Note that if the thrice repeated part of $\lambda$ is $2^km$ for some $k \geq 0$ and $m$ odd, then in $\mu$ there is a part equal to $2^km$ and the overlined part is of the form $2^tm$ for some $t>k$. Thus, we obtain  an overpartition   in $\OD(n)$. 

\begin{example} Let $$\lambda=(2^7\cdot 3, 2^6\cdot 3, 2^5\cdot 3, 2^4\cdot 3, 2^4\cdot 3, 2^4\cdot 3, 2\cdot 23, 9, 2\cdot 3, 5, 2\cdot 1)$$ $$= (384, 192, 96, 48^3, 46, 9, 6, 5, 2)\in \T(884).\ \ \ \ \ \ \ \ \ \ \ \ \ \ \ \ \ \ \ \    $$ Merging equal parts as in Glaisher's bijection, we obtain the partition $\mu=(768, 48, 46, 9, 6, 5, 2)=(2^8\cdot 3,  2^4\cdot 3, 2\cdot 23, 9, 2\cdot 3, 5, 2\cdot 1)\in \D(884)$. The smallest part of $\mu$ that is not a part of $\lambda$ is $768$. Thus, we obtain the overpartition $(\overline{768}, 48, 46, 9, 6, 5, 2)\in \OD(884)$.\medskip

Similarly, after applying Glaisher's bijection, the partition $$\lambda=(2^8\cdot 3, 2\cdot 23, 2^3\cdot 3,  2^2\cdot 3,  9, 2\cdot 3, 2\cdot 3, 2\cdot 3, 5, 2\cdot 1)$$ $$ = (768, 46, 24, 12, 9, 6^3, 5, 2)\in \T(884) \ \ \ \ \ \ \ \ \ \ \ \ \ \ \ \ \    $$ maps to $\mu = (768, 48, 46, 9, 6, 5, 2)=(2^8\cdot 3,  2^4\cdot 3, 2\cdot 23, 9, 2\cdot 3, 5, 2\cdot 1)\in \D(884)$. The smallest part of $\mu$ not appearing in $\lambda$ is $48$. Thus, we obtain  the overpartition $(768, \overline{48}, 46, 9, 6, 5, 2)\in \OD(884)$.\end{example} \medskip

\begin{remark} \label{rem} We could have obtained the  transformation above from the combinatorial proof of part (ii) of Theorem \ref{T1}. In the transformation from $\C(n)$ to $\MDS(n)$, we have $f=3$, $h=2$, and $N_h=1$. Thus $d=0$ and the decorated part is the smallest part in the transformed partition $\mu$ that did not occur in the original partition $\lambda$. Then $$r_h=1+k+\max\{j \mid  2^k\cdot m, 2^{k+1}\cdot m, \ldots, 2^{k+j}\cdot m \mbox { are all parts of } \lambda\}.$$ Thus, in $\mu$, the decorated part $2^{r_h}\cdot m$ is decorated with a bit string of consisting of all zeros and of length $r_h-k-1$, one less than the difference in exponents of $2$ of the decorated part and the next smallest part with the same largest odd factor $m$. Since the decoration of a partition in $\MDS(n)$ is completely determined by the part being decorated, we could simply just overline the part.  \end{remark}

\section{Extending Theorem \ref{T2} to an Euler type identity}

It is natural to look for the analogue of Theorem\ref{T1} (i) in the setting of Theorem \ref{T2}. If in Euler's partition identity we relax the condition in $\D(n)$ to  allow one part to be repeated exactly three times, how do we relax the condition on $\OO(n)$  to obtain an identity? We can search for the condition by following the proof of Theorem \ref{T1} part (i) but only for decorated partitions from $\MDS(n)$, were an even part is decorated with a bit string consisting entirely of zeros as in Remark \ref{rem}, i.e., of length one less than the difference in exponents of $2$ of the decorated part and the next smallest part with the same largest odd factor $m$. We identify these decorated partitions with the overpartitions in $\OD(n)$. Following the algorithm, we see that the set $\T(n)$ has the same cardinality as the set of partitions with exactly one even part $2^k\cdot m$, $k \geq 1, m$ odd, which appears with odd multiplicity. Moreover, part m appears with multiplicity at least $2^{k-1}$ and the multiplicity of $m$ must belong to an interval $[2^s-2^{k-1}, 2^s-1]$ for some $s\geq k$. Given the elegant description of the partitions in $\T(n)$ it would be desirable to find a nicer set of the same cardinality consisting of partitions with only one even part under some constraints. \medskip

Let $\A'(n)$ be the subset of $\A(n)$ consisting of partitions $\lambda$ of $n$ such that the set of even parts has exactly one element and satisfying the following two conditions: 

1) the even part $2^k \cdot m$, $k\geq 1, m$ odd,  has odd multiplicity and  

2) the largest odd factor $m$ of the even part is a part of $\lambda$ with multiplicity between $1$ and $2^k-1$.  \medskip

Let $a_1(n)=|\A'(n)|$. 

\begin{theorem} \label{T3} Let $n\geq 1$. Then, $a_1(n)=b_1(n)$. 

\end{theorem}

\begin{proof} \noindent \textit{From  $\OD(n)$ to $\A'(n)$:} 

Start with an overpartition $\mu\in\OD(n)$. Suppose the overlined part is $\mu_i=2^s\cdot m, s\geq 1, m$ odd.  Then there is a part $\mu_j=2^tm$ of $\mu$ with $0\leq t<s$. Keep part $2^s\cdot m$ and remove its overline. Split each part of the form $2^u\cdot m$ with $u>s$ (if it exists) into $2^{u-s}$ parts equal to $2^s\cdot m$. Split each part of the form $2^v\cdot m$ with $0\leq v<s$ into $2^v$ parts equal to $m$. Split every other even part into odd parts. Call the obtained partition $\lambda$. Then the multiplicity of  $2^s\cdot m$ in $\lambda$ is odd. Since there is a part $\mu_j=2^tm$ of $\mu$ with $0\leq t<s$, there will be at least one part equal to $m$ in $\lambda$. The largest possible multiplicity of $m$ in $\lambda$ is $2^{s-1}+2^{s-2}+\cdots + 2+1=2^s-1$. Thus $\lambda \in \A'(n)$. \medskip

\noindent \textit{From $\A'(n)$ to  $\OD(n)$:} 

Let $\lambda \in \A'(n)$. Merge equal terms repeatedly (as in Glaisher's bijection) to obtain a partition with distinct parts. Overline the part equal to the even part in $\lambda$. Call the obtained overpartition $\mu$. Since the even part $2^k\cdot m$ in $\lambda$ has odd multiplicity, there will be a part in $\mu$ equal to $2^k\cdot m$. Since $m$ has multiplicity between $1$ and $2^k-1$ in $\lambda$, there will be a part of size $2^i\cdot m$ in $\mu$ whenever there is a $1$ in position $i$ in the binary representation of the multiplicity of $m$ in $\lambda$. The binary representation of $2^k-1$ is a bit string of length $k-1$ consisting entirely of ones. Thus, in $\mu$ there is at least one part of size $2^t\cdot m$ with $0\leq t<k$ and $\mu\in \OD(n)$. 

\end{proof}

From Theorem \ref{T2} and \ref{T3} we obtain the following Euler type identity.

\begin{corollary} Let $n\geq 1$. Then, $a_1(n)=c_1(n)$. 

\begin{example} Let $n=10$. Then $$\T(20)=\{(7,1,1,1), (5,2,1,1,1), (4,3,1,1,1),(4,2,2,2),(3,2,2,2,1),(3,3,3,1)\}$$ and $$\A'(10)=\{(7,2,1),(3,2,2,2,1), (5,4,1), (3,4,1,1,1), (6,3,1), (8,1,1)\}.$$

\end{example}

\end{corollary}

\section{One part repeated exactly two times, all other parts distinct}

Given the specialization to Theorem \ref{T1} obtained in Theorems \ref{T2} and \ref{T3}, it is natural to ask what happens if one considers the set of partitions such that one part is repeated exactly two times and all other parts are distinct. Let $\S(n)$ be the subset of $\C(n)$ consisting of such partitions and let $c_2(n)=|\S(n)|$. We would like to express $c_2(n)$ as an excess of parts between partitions in $\D(n)$ and $\OO(n)$ (where parts are counted with different multiplicities) to obtain an identity similar to $c(n)=b(n)$ and $c_1(n)=b_1(n)$. \medskip

Note that $b(n)$ is  the difference between the number of parts in all partitions in $\OO(n)$ and the number of parts in all partitions in $\D(n)$. Thus, each part appearing in a partition in $\OO(n)$ is counted with the  multiplicity with which it appears in the partition. On the other hand, $b_1(n)$ is  the difference between the number of parts in all partitions in $\D(n)$ and the number of \textit{different} parts in all partitions in $\OO(n)$. Here, each part appearing in a partition in $\OO(n)$ is counted with multiplicity $1$ for that partition. 

\begin{definition} Given a partition $\lambda \in \OO(n)$, suppose the multiplicity of $i$ in $\lambda$ is $m_i$. If $i$ appears in $\lambda$, we define the \textit{binary order of magnitude} of the multiplicity of $i$ in $\lambda$, denoted $bomm_{\lambda}(i)$, to be the number of digits in the binary representation of $m_i$. 

\end{definition}

Note that, if $\m_i>0$, then  $bomm_{\lambda}(i)=\lfloor \log_2 (m_i)\rfloor +1$. \medskip

Let $b_2(n)$ denote  the difference between the number of parts in all partitions in $\OO(n)$,  each counted as many times as its $bomm$, and the number of parts in all partitions in $\D(n)$. Since the number of parts in all partitions in $\D(n)$ equals the number of $1$ in all binary representations of all multiplicities in all partitions of $\OO(n)$, it follows that $b_2(n)$ equals the number of $0$ in all binary representations of all multiplicities in all partitions of $\OO(n)$. We have the following theorem. 

\begin{theorem} Let $n\geq 1$. Then, $c_2(n)=b_2(n)$.

\end{theorem} 

\begin{proof} Let $\MDS'(n)$ be the subset of decorated partitions $\mu$ in $\MDS(n)$ such that a  part $2^sm$ of $\mu$ with $s\geq 1$ and $m$ odd can be decorated only if $2^{s-1}m$ is not a part of $\mu$.  The decoration $w$ must satisfy  $d_w=0$ and $0\leq \ell(w)\leq s-k-2$, where  $k=\max\{j<s \mid 2^jm \mbox{ is a part of } \mu\}$ if there is a part $2^jm$ in $\mu$ and $k=-1$ otherwise.  

Recall that, in a bit string, the most right position of a digit is position $0$ and we count positions from right to left. To see that $b_2(n)=|\MDS'(n)|$ we argue as follows. If there is a part $2^sm$ in $\mu\in \MDS'(n)$, with $s\geq 0$ and $m$ odd, then there is a $1$ in position $s$ of the binary representation of $m_m$, the multiplicity of $m$ in the partition $\lambda \in \OO(n)$ with $\varphi(\lambda)=\mu$. If the part $2^sm$ of $\mu$ is decorated, then $s\geq 1$ and $2^{s-1}m$ is not a part of $\mu$. Then,  digit $1$ in position $s$ of the binary representation of $m_m$ is followed by  $s-k-1$ zeros, where $k$ is defined as above. Note that in this case $k<s-1$ and therefore $s-k-1>0$. Since there are $s-k-1$ possible decorations for part $2^sm$, the total number of possible decorations of parts of $\mu$ equals the number of $0$ in the binary representations of all multiplicities of the corresponding $\lambda\in \OO(n)$.

Next, we show that $c_2(n)=b_2(n)$ by creating a one-to-one correspondence between $\S(n)$ and $\MDS'(n)$. \medskip

\noindent \textit{From $\MDS'(n)$ to $\S(n)$:}
 
Start with a decorated partition $\mu\in \MDS'(n)$. Suppose part $\mu_i=2^sm$, $s\geq 1$ and $m$ odd, is decorated with word $w$ with $d_w=0$ and $0\leq \ell(w)\leq s-k-2$, where $k$ is defined as above. Then, we split $2^sm$ into parts $2^{s-1}m, 2^{s-2}m, \ldots, 2^{s-\ell(w)}m$ and two parts equal to $2^{s-\ell(w)-1}m$. Note that if $\ell(w)=0$ then $2^s$ splits into two parts equal to $2^{s-1}m$. If there is a part $2^jm$ in $\mu$, then $k$ is defines as the largest $k<s$ such that $2^km$ is a part of $\mu$. Since $s-\ell(w)-1\geq k+1$, part $2^{s-\ell(w)-1}m$ appears exactly twice and all other parts appear once. The obtained partition is in $\S(n)$. 

\medskip

\noindent \textit{From $\S(n)$ to$\MDS'(n)$:}

Let $\lambda\in \S(n)$. Suppose the part that appears exactly twice is $2^km$ with $k \geq 0$ and $m$ odd. Merge the parts of $\lambda$ repeatedly using Glaisher's bijection $\varphi$ to obtain a partition $\mu$ with distinct parts. There will be no part equal to $2^km$ in $\mu$. The decorated part will be the only part of $\mu$ that is not a part of  $\lambda$. As in the proof of Theorem \ref{T1} (ii), $N_h=0$ and, since $f=2$, we have $d=\ds\frac{f-N_h}{2}-1=0$. Moreover, if the decorated part is $2^sm$ with $m$ odd, then $\ell(w)=s-k-1$. If there is a part $2^tm$ with $t<s$ in $\mu$, then, by construction, $t\leq k-1$. Then $\mu\in \MDS'(n)$. 

\end{proof}

In order to establish an Euler type identity let $\A''(n)$  be the subset of $\A(n)$ consisting of partitions $\lambda$ of $n$ such that the set of even parts has exactly one element and satisfying the following two conditions: 

1) the even part $2^k \cdot m$, $k\geq 1, m$ odd,  has odd multiplicity and  

2) the largest odd factor $m$ of the even part is a part of $\lambda$ with multiplicity between $0$ and $2^k-2$.  

\medskip

Let $a_2(n)=|\A''(n)|$. \medskip

Following the proof of Theorem \ref{T3}, one can show that $a_2(n)=b_2(n)$. Then, we have the following theorem. 

\begin{theorem} Let $n\geq 1$. Then, $a_2(n)=b_2(n)=c_2(n)$. 

\end{theorem}

\begin{example} Let $n=10$. Then $$\S(10)=\{(8,1,1), (6,2,1,1), (5,3,1,1), (6,2,2), (4,2,2,1), (4,3,3), (4,4,1), (5,5)\}$$ and $$\A''(10)=\{(5,3,2), (2,2,2,2,2), (4,3,3), (5,4,1), (6, 3, 1), (6, 1, 1, 1, 1), (8,1,1), (10)\}.$$

\end{example}

\bigskip


\end{document}